\documentclass[a4paper]{article}
\usepackage{amsmath,amssymb,amsthm,hyperref}
\numberwithin{equation}{section}
\newtheorem{thm}[equation]{Theorem}

\newtheorem{cor}[equation]{Corollary}
\newtheorem{lem}[equation]{Lemma}

\newcounter{mycount}
\newenvironment{romlist}{\begin{list}{\rm(\roman{mycount})}
   {\usecounter{mycount}\labelwidth=1cm\itemsep 0pt}}{\end{list}}

\def\EE{{\mathbb E}}

\def\NN{{\mathbb N}}

\def\RR{{\mathbb R}}
\def\ZZ{{\mathbb Z}}
\def\CC{{\mathbb C}}

\def\a{\alpha}
\def\dd{{\,\mathrm d}}
\def\b{\beta}

\def\t{\tau}

\def\om{\omega}

\def\be{\begin{equation}}
\def\ee{\end{equation}}
\def\e{e}

\def\O{{\mathrm O}}
\def\oo{\infty}

\def\q{\quad}
\def\qq{\qquad}

\def\bc{\beta_\text{c}}
\def\bt{\beta_\tau}
\def\fl2#1{{\lfloor#1/2\rfloor}}
\def\cl2#1{{\lceil#1/2\rceil}}
\def\X{X}
\def\saw{self-avoiding walk}

\begin{document}
\title{Borel type bounds for the self-avoiding walk connective constant}
\author{B.\ T.\ Graham\\
   {\small DMA--\'Ecole Normale Sup\'erieure, 45 rue d'Ulm}\\
   {\small 75230 Paris Cedex 5, France}\\
   {\small \tt graham@dma.ens.fr}
}
\maketitle
\begin{abstract}
\noindent Let $\mu$ be the self-avoiding walk connective constant on $\ZZ^d$. We show that the asymptotic expansion for $\bc=1/\mu$ in powers of $1/(2d)$ satisfies Borel type bounds. This supports the conjecture that the expansion is Borel summable.
\end{abstract}
\section{Introduction}
Let $\ZZ^d$ denote the hypercubic lattice, with nearest neighbour edges.
A \saw\ of length $n$ is a sequence of points $\om_0,\om_1,\dots,\om_n$ in $\ZZ^d$ such that $|\om_{i}-\om_{i+1}|=1$ and for $i\not=j$, $\om_i\not=\om_j$. Let $c_n$ denote the number of self-avoiding walks, up to translation invariance, of length $n$ on $\ZZ^d$. It is well known that the limit $\mu(d)=\lim_{n\to\oo} c_n^{1/n}$ exists; the limit is called the connective constant.
Fisher and Gaunt calculated that \cite{FisherGaunt}
\[
\mu=2d-1-1/(2d)-3/(2d)^2-16/(2d)^3-102/(2d)^4-\dots
\]
However, it is not clear from their calculation that the expansion can be continued to higher orders. Moreover, in low dimensions numerical extrapolation techniques are needed to use the expansion to estimate $\mu$. 

The \saw\ is most easily understood in high dimensions. As $d\to\oo$, paths with large loops become relatively rare. It is therefore useful to consider a walk with only local self-avoidance. Say that $\om_0,\dots,\om_n$ is a memory-$\tau$ \saw\ if $\om(i)\not=\om(j)$ for $0<|i-j|\le\tau$.
Let $c_n^{(\tau)}$ denote the number of $n$-step memory-$\tau$ walks, up to translation invariance, and let $\mu_\tau(d)=\lim_{n\to\oo} (c_n^{(\tau)})^{1/n}$.
Using memory-4 self-avoiding walks as a starting point (taking into account loops of size 2 and 4) Kesten showed that \cite{KestenSAWii}
\[
\mu(d)=2d-1-1/(2d)+\O(1/d^2).
\]
The lace expansion is a powerful technique for exploring the properties of the \saw\ in dimensions $d>4$; we refer the reader to \cite{SladeLectureNotes} for a recent introduction.
We will look at the quantity $\bc=1/\mu(d)$; $\beta_c$ is the radius of convergence of the Lace expansion generating function. For convenience, let $s=1/(2d)$. A series expansion exists for $\bc$ in powers of $s$  \cite{HaraSladeHighDExpansions}
\[
\bc(s)=\sum_{n=1}^\oo \a_n s^n=s+s^2+2s^3+ 6s^4+27s^5+157s^6+\dots 
\]
Expansions in powers of $1/d$ have been developed for many other models in statistical physics, such as the Ising model \cite{FisherGaunt}, percolation \cite{GauntRuskin}, lattice animals and the $n$-vector model \cite{GerberFisher}. Finding the coefficients of the expansion is normally computationally intensive. It is often even more difficult to determine the basic properties of the expansion. What is the radius of convergence? Is it an asymptotic expansion?  Can the expansion be interpreted as a Borel sum?

In the case of the \saw\ it is not known, but it is widely believed, that the radius of convergence is zero. Using the Lace expansion, it has been shown that the series expansion is asymptotic \cite{HaraSladeHighDExpansions}. 
We will show that the partial sums satisfy Borel type bounds. Borel summability raises the prospect of calculating $\bc$ from the series expansion even if the radius of convergence is zero. 
\begin{thm}\label{SAWbound} There exist a constant $C_1$ such that for all $d$,
\begin{eqnarray}\label{SAWboundEqn}
\left|\bc(s)-\sum_{n=1}^{M-1} \a_n s^n\right| \le C_1^M s^M M!, \q M=1,2,\dots
\end{eqnarray}
\end{thm}
\noindent The motivation for Theorem \ref{SAWbound} is discussed in Section \ref{sec:borel}. 
In Section \ref{sec:reversion}, we use the Lace expansion to derive a formula for the $\a_n$. This formula is used in Section \ref{factorial_bounds} to control the growth of the $\a_n$ as $n\to\oo$. In Section \ref{sec:diag}, we consider the diagrammatic estimates for the Lace expansion. Finally, in Section \ref{sec:proof}, we prove Theorem \ref{SAWbound}.

\section{Borel summability and the spherical model\label{sec:borel}}
In light of Theorem~\ref{SAWbound}, it is natural to ask if $\bc$ can be recovered from the $\a_n$ by means of a Borel sum. Let $B$ denote the Borel transform of the asymptotic expansion for $\bc$; $B$ is well defined (see Lemma \ref{bound_an}) in a neighbourhood of zero by
\[
B(t)=\sum_{n=1}^\oo \a_n t^n / n!
\]
We conjecture that $B$ can be extended analytically to a neighbourhood of the positive real axis, and that $\bc(s)$ is equal to the Borel sum
\[
\sum_\text{Borel} \a_n s^n:=\frac{1}{s}\int_0^\oo e^{-t/s} B(t)\dd t.
\]
There are two reasons for making this conjecture.
Firstly, with $R>0$, let $C_R:=\{z\in\CC:\text{Re } z^{-1}>R^{-1}\}$ denote the open disc in $\CC$ with centre $R/2$ and diameter $R$ \cite[Figure 1]{SokalBorel}. Suppose that $\bc$ can be extended to an analytic function on $C_R$ such that \eqref{SAWboundEqn} holds for all $s\in C_R$. Under this assumption, the Borel sum is well defined in $C_R$ and equal to $\bc$ \cite{SokalBorel}. Unfortunately, it is not clear how to extend $\bc$ to an analytic function on $C_R$. Interpreting the Borel sum remains an open problem. 

Secondly, there is the case of the spherical model \cite{GerberFisher}. Physicists consider both the \saw\ and the spherical model to be special cases of the $n$-vector model. The spherical model is an exactly solvable spin system with dependence between the spins. Let $I_0$ denote the $0$th-order modified Bessel function of the first kind. The critical point is
\[
K_c(d)=\frac{1}{2} \int_0^\oo e^{-g(x)d} \dd x, \qq g(x)=x-\log I_0(x).
\]
Function $g$ is monotonic, $g(0)=0$ and $g(x)\to\oo$ as $x\to\oo$. Let $(a_n)$ denote the sequence with generating function $g$-inverse,
\[
g^{-1}(t)=\sum_{n=1}^\oo a_n t^n.
\]
Integration by parts yields an asymptotic expansion for $K_c(d)$ in powers of $1/d$,
\[
K_c(d)=\frac12\sum_{n=1}^\oo \frac{a_n n!}{d^n}.
\]
We have said that very little is known rigorously about $1/d$ expansions. This is the exception that proves the rule.
The radius of convergence is zero, but the expansion can be interpreted as a Borel sum with Borel transform $g^{-1}/2$.

Note that the $(a_n)$ oscillate.
The first $12$ coefficients are positive, the next $8$ are negative; the pattern of signs goes 
\[
12,\  8,\   9,\   9,\   9,\   9,\   9,\   9,\   9,\  10,\   9,\   9,\   9,\   9,\   9,\   9,\   9,\   9,\   9,\   9,\   9,\   9,\  10,\   9,\   9,\   9,\   9,\   9, \dots
\]
This oscillation is related to the fact that $g^{-1}$ has no poles on the positive real axis. The coefficients $\a_n$ for the \saw\ also show a change of sign. For $i=1,2,\dots,11$, $\a_i$ is positive; the remaining coefficients that are known, $\a_{12}$ and $\a_{13}$, are negative \cite{SladeEnumeration}.

\section{The expansions $\bt=\sum_{n=1}^\oo \a_n s^n$\label{sec:reversion}}
The starting point in our analysis will be \cite[(2.2)]{HaraSladeHighDExpansions}. Let $\bt=1/\mu_\tau$, and take $\b_\oo=\bc$. When $d$ is sufficiently large, for $\t$ finite and infinite,
\be\label{identity}
\bt=s(1-\hat{\Pi}_{\bt}(0;\tau)).
\ee
The quantity $\hat{\Pi}_{\bt}(0;\tau)$ is the Fourier transform of the memory-$\tau$ Lace expansion at $k=(0,\dots,0)$ and $\b=\bt$; essentially, it describes the difference between the generating functions of simple random walk and \saw.
The Lace expansion can be thought of as a series of inclusion/exclusion terms. 
A type-$N$ lace graph is a simple walk with at least $N$ points of self-intersection, and certain additional constraints. Let $\pi^{(N)}_a(x;\t)$ count the number of memory-$\t$ lace graphs from $0$ to $x$ of type $N$ and length $a$. Then
\[
\hat{\pi}^{(N)}_a(k;\tau)=\sum_{x\in\ZZ^d}  \pi^{(N)}_a(x;\tau) e^{-i k\cdot x},\qq 
\hat{\Pi}^{(N)}_\b(k;\tau)=\sum_{a=N+1}^\oo \hat{\pi}^{(N)}_a(k;\t) \b^n,
\]
\[
\text{and finally}\qq \hat{\Pi}_{\b}(k;\tau)= \sum_{N=1}^\oo (-1)^N \hat{\Pi}^{(N)}_\b(k;\tau) .
\]
The definition of a lace graph respects the symmetries of the underlying lattice.
There are $2^d d!$ ways of choosing an ordered orthonormal basis for $\RR^d$ from the set $\ZZ^d$. 
Each lace graph in $\ZZ^d$ with dimensionality $D$ is equivalent to $2d(2d-2)\dots(2d-2D+2)$ other lace graphs under the action of this group of symmetries. 

Let $f_\t(a,N,D)$ count the number of equivalence classes of memory-$\t$ lace graphs in $\ZZ^D$ with length $a$, type $N$ and dimensionality $D$. Lace graphs with dimensionality $D$ have length at least $2D$. Therefore, we can write the number of memory-$\t$ lace graphs of length $a$ and type $N$ in $\ZZ^d$ as a polynomial in powers of $s^{-1}=2d$,
\[
\sum_{D=1}^\fl2{a} f_\t(a,N,D)\, 2d(2d-2)\dots(2d-2D+2)=\sum_{b=\cl2{a}}^{a-1} c_{a,b,N} s^{b-a}.
\]
Let $I=\{(a,b):b=1,2,\dots;a=b+1,\dots,2b\}$ and set
\[
c_{a,b}=\sum_{N=1}^\oo (-1)^{N+1} c_{a,b,N},\qq (a,b)\in I.
\]
The $c_{a,b}$ depend implicitly on $\t$, but $c_{a,b}$ is fixed once $\t\ge a$.
Using this notation to rewrite \eqref{identity} yields a formal power series,
\begin{eqnarray}\label{iterative_formula}
\bt&=&s\Big[1+\sum_{(a,b)\in I} \bt^a c_{a,b} s^{b-a}\Big]\\
   &=&s[1+\b_\t^2 c_{2,1} s^{-1}+\b_\t^3 c_{3,2} s^{-1}+\b_\t^4 (c_{4,3} s^{-1}+c_{4,2} s^{-2})+\dots]\nonumber
\end{eqnarray}
Plugging ``$\bt=0$'' into the right hand side gives ``$\bt=s$''. Taking ``$\bt=s$'' and plugging it back into the right hand side then gives ``$\bt=s+c_{2,1}s^2+(c_{3,2}+c_{4,2})s^3+\dots$''; iterating in this way yields a series expansion for $\bt$:
\[
\bt= s+c_{2,1}s^2 + (2c_{2,1}^2+c_{3,2}+c_{4,2})s^3 + \dots
\]
The coefficient $\a_n$ of $s^n$ in $\bt$ is thus a multinomial function of the $c_{a,b}$.
\begin{lem}\label{an_bound}
With $S_n:=\{(n_{a,b})\in\NN^I:n=1+\sum_I b n_{a,b}\}$, 
\[
\a_n=\sum_{(n_{a,b})\in S_n}
\frac{[\sum_I a n_{a,b}]!}{[\,\prod_I n_{a,b}!][1+\sum_I (a-1)n_{a,b} ]!}
\prod_I c_{a,b}^{n_{a,b}}. 
\]
\end{lem}
\noindent It is a corollary of Lemma~\ref{an_bound} that $\a_n$ only depends on $\t$ if $\t<2n-2$. 
\begin{proof}[Proof of Lemma~\ref{an_bound}]
Let $\phi(\b)=1+\sum_I c_{a,b} \b^a s^{b-a}$.
Setting $\b=\sum_{n=1}^\oo \a_n s^n$, \eqref{iterative_formula} becomes
\[
\frac{\b}{\phi(\b)}=s.
\]
We will use the notation $[\b^n]\phi(\b)$ to denote the coefficient of $\b^n$ from generating function $\phi(\b)$. 
The Lagrange--B\"urmann series reversion formula states that,
\[
\b=\sum_{k=1}^\oo \frac{s^k}{k}  [\b^{k-1}]\phi(\b)^k.
\]
Applying the formula yields,
\begin{eqnarray*}
\sum_{n=1}^\oo\a_n s^n&=&\sum_{k=1}^\oo \frac{s^k}{k} [\b^{k-1}]\biggl(1+\sum_I c_{a,b} \b^a s^{b-a}\biggr)^k
\end{eqnarray*}
and so with $T_k:=\{(n_{a,b})\in \NN^I: k-1=\sum_I an_{a,b}\}$,
\[
\a_n=[s^n] \sum_{k=1}^\oo \frac{s^k}{k} 
\sum_{(n_{a,b})\in T_k}
 \frac{k!}{(k-\sum_I n_{a,b})!\prod_I [n_{a,b}!]} \prod_I\Big[ c_{a,b} s^{b-a}\Big]^{n_{a,b}}.
\]
Extracting the coefficient of $s^n$ leaves only the terms with $n=k+\sum_I (b-a)n_{a,b}$. 
By the definition of $T_k$, these are the terms with $(n_{a,b})\in S_n$.
\end{proof}

\section{Factorial bounds on $(\a_n)$\label{factorial_bounds}}
We can use Lemma~\ref{an_bound} to bound the coefficients $(\a_n)$ of the asymptotic expansion.
\begin{lem}\label{bound_an}
There is a constant $C_2$ such that for all $n$, $|\a_n|\le  C_2^n n!$
\end{lem}
\noindent This is achieved by bounding $|c_{a,b}|$ in terms of $b$.
\begin{lem} \label{bound_cab}
Let $c_b=\sum_{a=b+1}^{2b}|c_{a,b}|$. There is a constant $C_3$ such that $c_b\le C_3^b b!$
\end{lem}
\begin{proof}
The numbers $(c_{a,b})$ are defined in terms of lace graphs with length $a$,
\[
|c_{a,b}|\le \sum_{D=a-b}^\fl2{a} \sum_{N=1}^\oo f_\t(a,N,D)\times\Big|[s^{b-a}] s^{-1}(s^{-1}-2)\dots(s^{-1}-2D+2)\Big|.
\]
The number of lace graphs of length $a$ in $\ZZ^D$ is at most $(2D)^a$, so 
\[
\sum_{N=1}^\oo f_\t(a,N,D) \le \frac{(2D)^a}{ 2^D D!}.
\]
The absolute value of $[s^{b-a}]s^{-1}(s^{-1}-2) \dots(s^{-1}-2D+2)$ is at most
\[
[s^{b-a}](s^{-1}+2D)^D= (2D)^{D+b-a} \binom{D}{a-b},\q D=a-b,\dots,\fl2{a}.
\]
The result follow by Stirling's formula: for some constant $C_3$,
\begin{eqnarray*}
\sum_{a=b+1}^{2b} |c_{a,b}|&\le& \sum_{a=b+1}^{2b} \frac{2^b}{(a-b)!}\sum_{D=a-b}^\fl2{a} \frac{D^{D+b}}{(D+b-a)!}\\
         &\le& \sum_{a=b+1}^{2b}\frac{2^b}{(a-b)!} (1+\fl2{a} -(a-b)) \frac{b^{2b}}{(2b-a)!}\le C_3^b b!\qedhere
\end{eqnarray*}
\end{proof}
\noindent Before we can prove Lemma~\ref{bound_an}, we need a bound on how a power series with factorial coefficients behaves under exponentiation.
\begin{lem}\label{factorial_generating_fn} 
Let $\phi(\b)\equiv \sum_{k=0}^\oo k! \b^k$. Then
\[
[\b^k] \phi(\b)^n \le k! \prod_{j=1}^k (1+(n-1)/j^2)
\le 6^n k!
\]
\end{lem}
\begin{cor}\label{factorial_cor}
Let $\psi(\b)\equiv \sum_{k=1}^\oo k! \b^k$. Comparing $\psi(\b)$ with $\b\phi(2\b)$ gives
\[
[\b^k] \psi(\b)^n \le 6^k (k-n)!
\]
\end{cor}

\begin{proof}[Proof of Lemma~\ref{factorial_generating_fn}]
Let $l_1,\dots,l_n$ denote non-negative integers. The first inequality is equivalent to,
\be\label{equiv1}
\sum_{l_1+\dots+l_n=k}\ \ \prod_{i=1}^n\ l_i! \le k! \prod_{j=1}^k (1+(n-1)/j^2).
\ee
We will show this by induction in $k$. 
The second inequality is a simple consequence of the fact that $\sum_{j=1}^\oo 1/j^2 = \pi^2/6 < \log 6$.

For convenience, \eqref{equiv1} can be written in terms of a multinomial random variable $\X^k\equiv(X^k_1,\dots,X^k_n)\sim \text{Multinomial}(k;1/n,\dots,1/n)$:
\[
n^k \EE\left(\binom{k}{\X^k}^{-2}\right) \le \prod_{j=1}^k (1+(n-1)/j^2), \qq \binom{k}{\X^k}=\frac{k!}{X^k_1!\dots X^k_n!}.
\]
For the inductive step, we construct $\X^{k+1}$ from $\X^k$ by adding $1$ to one of $X_1^k,\dots,X_n^k$ uniformly at random. 
Let $\e_1=(1,0,0,\dots)$, $\e_2=(0,1,0,\dots)$, and so on.
The inductive step is then,
\begin{eqnarray*}
n\EE\left[\binom{k+1}{\X^{k+1}}^{-2}\right] &=& \EE\sum_{i=1}^n  \binom{k+1}{\X^k+\e_i}^{-2}\\
 &=& \EE\sum_{i=1}^n \left( \frac{k+1}{X^k_{i}+1}\right)^{-2}\binom{k}{\X^k}^{-2}\\
 &\le& \frac{(k+1)^2+(n-1)}{(k+1)^2} \EE\left[\binom{k}{\X^k}^{-2}\right].
\end{eqnarray*}
The inequality in the last step is the result of replacing $\sum_{i=1}^n \big( \frac{k+1}{X^k_i+1}\big)^{-2}$ with its supremum over the range of $\X^k$.
\end{proof}
\begin{proof}[Proof of Lemma~\ref{bound_an}]
For $(a,b)\in I$, $a\le 2b$. By Lemma~\ref{an_bound}, 
\begin{eqnarray*}
|\a_n| &\le& \sum_{(n_{a,b})\in S_n}\frac{(\sum a n_{a,b})!}{\prod n_{a,b}!(1+\sum(a-1)n_{a,b})!}\prod |c_{a,b}|^{n_{a,b}}\\
&\le& \sum_{(n_{a,b})\in S_n}\frac{(2n-2)!}{\prod n_{a,b}!(2n-1-\sum n_{a,b})!}\prod |c_{a,b}|^{n_{a,b}}\\
&\le& \frac1{2n-1}[\b^{n-1}]\left(1+\sum_b c_b \b^b\right)^{2n-1}.
\end{eqnarray*}
By Lemma~\ref{bound_cab} and Lemma~\ref{factorial_generating_fn}, $|\a_n| \le 6^{2n}  C_3^n n!$
\end{proof}

\section{Diagrammatic estimates\label{sec:diag}}
There is a number $C_\text{HS}$ \cite{HaraSladeHighDExpansions} such that for $d$ sufficiently large, for all $\t$, 
\be\label{C_HS}
\hat{\Pi}_{\bt}^{(N)}(0;\t)\le(sC_\text{HS})^N.
\ee
This is a consequence of the diagrammatic estimates \cite[Theorem~4.1]{SladeLectureNotes} for the number of type-$N$ lace graphs. The estimates can also be used to bound above the number of lace graphs according to length. 
\begin{lem}\label{diagram_estimates}
There is a constant $C_4$ such that,
\[
\hat{\pi}_a^{(N)}(0;\t) \le  [\b^a] \left(\sum_{n=1}^{\t/2} C_4^n s^{-n} n! \b^{2n}(1+s/\b)\right)^N.
\]
\end{lem}
\begin{proof}
The loops of a memory-$\t$ lace graphs have length at most $\t$. 
Define a generating function by 
\[
F_\b^\t=\sum_{k=1}^\t \b^k\sup_{x\in\ZZ^d} k c^{(0)}_k(x),
\]
where $c^{(0)}_k(x)$ denotes the number of \saw{s} from $0$ to $x$ of length $k$.
The diagrammatic estimates, see for example \cite[Theorem~4.1]{SladeLectureNotes}, imply that
\[
\hat{\pi}_a^{(N)}(0;\t)\le [\b^a]\, (F_\b^\t)^{N}.
\]
Let $C_4=20$.  The result follows when we show that for $k=1,\dots,\t$,
\[
\sup_{x\in\ZZ^d} k c^{(0)}_k(x) \le [\b^k] \sum_{n=1}^{\t/2} C_4^n s^{-n} n! \b^{2n}(1+s/\b).
\]
When $k$ is even, $c_k^{(0)}(x)$ is maximized by $x=0$.
First suppose $k=2m\le 2d$. The number of simple walks from $0$ to $0$ in $\ZZ^d$ of length $2m$ is less than the number of $2m$ step walks that are restricted to an $m$ dimensional subspace \cite{HammersleySankhya}: 
\[
c^{(0)}_{2m}(0)\le \binom{d}{m} (2m)^{2m} \le \frac{d^m}{m!}(2m)^{2m}.
\]
If $k=2m>2d$, $c^{(0)}_{2m}(x)\le (2d)^{2m}$. By Stirling's formula,
\[
\forall m, \qq (2m) c^{(0)}_{2m}(0)\le C_4^m s^{-m}m! 
\]
When $k$ is odd, $c_k^{(0)}(x)$ is maximized by $x=e_1=(1,0,\dots,0)$ and
\[
c^{(0)}_{2m-1}(e_1)= c^{(0)}_{2m}(0)/(2d).\qedhere
\]
\end{proof}

\section{Proof of Theorem~\ref{SAWbound}\label{sec:proof}}
It is well known that $d\le \mu\le 2d$, and hence that $s\le \bc\le 2s$. Let $C_5\in[0,C_2^{-1}]$ and suppose that $C_1 \ge10 C_2/C_5$.
Then for $s\ge C_5/M$, inequality \eqref{SAWboundEqn} holds simply by Lemma~\ref{bound_an}.
We will show by induction that for some constants $C_5$ and $C_6$, for $k=1,2,\dots$ and $s\le C_5/k$,
\be\label{C8}
\bc=\sum_{n=1}^{k-1} \a_n s^n+E_k s^k \text{ with } |E_k|\le C_6^k k!
\ee
Theorem~\ref{SAWbound} follows from \eqref{C8} by taking $C_1= \max\{C_6,10 C_2/C_5\}$.

To begin the induction process, notice that $E_1=\bc/s\in[0,2]$, so \eqref{C8} holds for $k=1$ if $C_6\ge 2$.
To keep the remainder of the proof simple, we will now assume that $C_6\gg 1$, so we can take it for granted that $C_6\ge 2$.  We will also assume that $C_5 \ll C_6^{-3}\ll 1$.

Assume inductively that \eqref{C8} holds for $k=1,\dots,M$.
Let $s\le C_5/(M+1)$. We need to show that
\[
|E_{M+1}| \le C_6^{M+1} (M+1)!
\]
We will define $(A_i)_{i=1}^4$ such that
\[
E_{M+1}s^{M}=\sum_{i=1}^4 A_i, \q |A_i| \le \tfrac14s^{M}C_6^{M+1} (M+1)!
\]
We will use the Lace expansion. A type-$N$ memory-$\t$ lace graph has length at most $N\t$, so it is easier to use the finite memory version of the Lace expansion.
If the memory $\t$ is high enough, we can use $(\a_n)_{n=1}^M$ to refer to the first $M$ coefficients of the series expansions for both $\bc$ and $\bt$. 
Let $\t=2M$. By \cite[Theorem 1]{KestenSAWii}, there is a constant $C_\text{K}$ such that
\[
\forall s\le 1/(52M),\ 0\le \bc-\bt \le  s^{M+2} C_\text{K}^M M!
\]
Let $A_1=\bc-\bt$, and let $E^\t_M=E_M-A_1$ so
\be\label{bt}
\bt=\sum_{n=1}^{M-1}\a_n s^n + E^\t_M s^M. 
\ee
By \eqref{identity}, we must now choose $A_2,A_3,A_4$ such that
\be\label{bt2}
A_2+A_3+A_4+\sum_{n=1}^M \a_ns^{n-1}=
1- \hat{\Pi}_{\bt}(0;\t).
\ee
With reference to the bound \eqref{C_HS}, let 
\[
A_2=- \sum_{N=M+1}^\oo (-1)^N \hat{\Pi}_{\bt}^{(N)}(0;\t).
\]
Let $A_3$ match the terms generated on the right hand side of \eqref{bt2} by lace graphs of length $a\ge 2M$ and type $N\le M$,
\[
A_3 = - \sum_{N=1}^M (-1)^N \sum_{a=2M}^{2MN} \hat{\pi}^{(N)}_a(0;\t) \bt^a.
\] 
By Lemma~\ref{diagram_estimates},
\begin{eqnarray*}
|A_3|&\le& \sum_{N=1}^M \sum_{a=2M}^{2MN} \bt^a [\b^a]  \left(\sum_{n=1}^{M} C_4^n s^{-n} n!  \b^{2n}(1+s/\b)\right)^N
\\
&\le& \sum_{N=1}^M (1+s/\bt)^N \sum_{a=2M}^{2MN} \bt^a [\b^a]  \left(\sum_{n=1}^{M} C_4^n s^{-n} n!  \b^{2n}\right)^N.
\end{eqnarray*}
Setting $x=\b^2$, the right hand side above is equal to 
\begin{eqnarray*}
\sum_{N=1}^M (1+s/\bt)^N \sum_{a=M}^{MN} (C_4M\bt^2 s^{-1})^a [x^a] \left( \sum_{n=1}^{M} n!  x^n/M^n\right)^N.
\end{eqnarray*}
We can assume that $(1+s/\bt)C_4M\bt^2s^{-1}\le C_6e^{-1}Ms\le 1$. The above is less than
\begin{eqnarray*}
&&\sum_{N=1}^M\ \sum_{a=M}^{MN} (C_6e^{-1}M s)^{a} [x^a] \left( \sum_{n=1}^{M} x^{n} n!/M^n\right)^N\\
&&\le (C_6e^{-1}Ms)^M\sum_{N=1}^M \left( \sum_{n=1}^{M} n!/M^n\right)^N\le \tfrac14 s^{M} C_6^{M+1} (M+1)!
\end{eqnarray*}
By the process of elimination, $A_4$ is now defined by
\be\label{bt3}
A_4+\sum_{n=1}^M \a_ns^{n-1}=1-\sum_{N=1}^M(-1)^N\sum_{a=2}^{2M-1} \hat{\pi}^{(N)}_a(0;\t) \bt^a.
\ee
Substitute \eqref{bt} into the right hand side of \eqref{bt3}; by Lemma~\ref{an_bound}, we can cancel the powers of $s$ below $s^M$:
\[
A_4 = - \sum_{n=M}^\oo s^n [s^n] \sum_{N=1}^M (-1)^N \sum_{a=2}^{2M-1}\hat{\pi}^{(N)}_a(0;\t)\left(\sum_{n=1}^{M-1}\a_n s^n + E^\t_M s^M\right)^a.
\] 
Recall that $c_b := \sum_{a=b+1}^{2b} |c_{a,b}|$. As $\a_1=1$, 
\begin{eqnarray*}
|A_4|&\le& \sum_{n=M}^\oo s^n [s^n]\ \sum_{a=2}^{2M-1}\sum_{b=\cl2{a}}^{a-1} |c_{a,b}|s^{b-a} \left( \sum_{k=1}^{M-1}|\a_k| s^k + |E^\t_M| s^M\right)^{a}\\
     &\le&\sum_{n=M}^\oo s^n [s^n]\ \sum_{b=1}^{2M-2} c_{b}s^{-b} \left( \sum_{k=1}^{M-1}|\a_k| s^k + |E^\t_M| s^M\right)^{2b}.  
\end{eqnarray*}
The $\a_n$ are controlled by Lemma \ref{bound_an}, and the $c_{b}$ are controlled by Lemma \ref{bound_cab}.
By the inductive assumption and the bound on $|A_1|$, we have $|E^\t_M |\le 2C_6^M M!$ 
Expand the $(\,\cdot\,)^{2b}$ terms above: we will split the resulting terms into three groups.
\begin{romlist}
\item The terms $s^n c_{b}|\a_{i_1}||\a_{i_2}|\cdots|\a_{i_{2b}}|$ with $M\le n\le 2M-2$.
\item The terms $s^n c_{b}\dots$ with $M\le n\le 2M-2$ that are not in group (i) because they contain an $|E^\t_M|$ term.
\item The terms $s^n c_{b}\dots$ with $n\ge 2M-1$.
\end{romlist}
By Lemma~\ref{bound_an}, the contribution from the first group is less than
\be\label{sum_1}
\sum_{n=M}^{2M-2} s^n C_2^{n+1} (n+1)!  
\ee
The contribution of the second group is at most
\begin{eqnarray}\label{sum_2}
2 C_6^M M! s^M \sum_{n=0}^{M-2} s^n [s^n] \sum_{b=1}^{n+1} C_3^b b! s^{-b} \times (2b) \left(\sum_{k=1}^{M-1}C_2^k k! s^k \right)^{2b-1}.
\end{eqnarray}
For $k=1,\dots,M$, define $\hat{\a}_k$ by
\[
\hat{a}_ks^k=|\a_k|s^k+\dots+|\a_{M-1}|s^{M-1}+|E^\t_M|s^M.
\]
The contribution of the third group is at most
\begin{eqnarray}\label{sum_3}
s^{2M-1} [x^{2M-1}] \sum_{b=1}^{2M-2} C_3^b b! x^{-b} \left( \sum_{k=1}^{M}\hat{\a}_k x^k \right)^{2b}.
\end{eqnarray}
By having the $\hat{\a}_kx^k$ in the $(\,\cdot\,)^{2b}$ term, we catch all the terms in the third group while extracting only the coefficient of $x^{2M-1}$. [For example if $M=10$, the term $s^{28}c_{2}\a_6\a_7\a_8\a_9$ is accounted for by the term $s^{19}c_{2}\hat{\a}_6\hat{\a}_7\hat{\a}_7\hat{\a}_1 =$
\[
s^{19} c_2 (|\a_6|+ \dots)(|\a_7|+\dots)(\dots+|\a_8|s+\dots)(\dots+|\a_9|s^8+\dots) 
\]
generated by \eqref{sum_3}.] We can assume that $\hat{\a}_k \le (2C_6)^k k!$

The absolute value of $A_4$ is now bounded by the sum of three intimidating expressions, \eqref{sum_1}--\eqref{sum_3}. However, under suitable conditions on $C_5$ and $C_6$, they are quite easily seen to be small when compared to $\tfrac14 s^M C_6^{M+1}(M+1)!$ In the case of \eqref{sum_1}, $C_6\gg C_2$ and $C_2C_5\ll 1$ is sufficient. Corollary \ref{factorial_cor} allows us to control \eqref{sum_2} and \eqref{sum_3}. If $C_2^2C_3 C_5\ll 1$ and $C_6\gg C_2C_3$, \eqref{sum_2} is small. If $C_3^2 C_5C_6^3\ll 1$, \eqref{sum_3} is small.
This completes the proof of Theorem \ref{SAWbound}.

\section*{Acknowledgement}
The author would like to thank Gordon Slade for many helpful discussions and the Fondation Sciences Math\'ematiques de Paris for financial support.

\bibliography{borel}
\bibliographystyle{plain}

\end{document}